\documentclass[12pt]{amsart}

\usepackage{hyperref}
\usepackage{tikz}

\usetikzlibrary{positioning}

\newtheorem{thm}{Theorem}[section]
\newtheorem{lemma}[thm]{Lemma}
\newtheorem{prop}[thm]{Proposition}
\newtheorem{cor}[thm]{Corollary}

\theoremstyle{definition}

\theoremstyle{remark}
\newtheorem*{claim}{Claim}

\newcommand{\Homeo}{\mathop{\rm Homeo}}
\newcommand{\Diff}{\mathop{\rm Diff}}
\newcommand{\id}{\mathop{\rm id}}

\begin{document}

\author{Michael P. Cohen}
\address{Michael P. Cohen,
Department of Mathematics,
North Dakota State University,
PO Box 6050,
Fargo, ND, 58108-6050}
\email{michael.cohen@ndsu.edu}

\title{Polishability of some groups of interval and circle diffeomorphisms}

\begin{abstract}  Let $M=I$ or $M=\mathbb{S}^1$ and let $k\geq 1$.  We exhibit a new infinite class of Polish groups by showing that each group $\Diff_+^{k+AC}(M)$, consisting of those $C^k$ diffeomorphisms whose $k$-th derivative is absolutely continuous, admits a natural Polish group topology which refines the subspace topology inherited from $\Diff_+^k(M)$.  By contrast, the group $\Diff_+^{1+BV}(M)$, consisting of $C^1$ diffeomorphisms whose derivative has bounded variation, admits no Polish group topology whatsoever.
\end{abstract}

\maketitle


\section{Introduction}

We are motivated by the following question: let $M$ be a compact one-dimensional manifold (i.e. the interval $I$ or the circle $\mathbb{S}^1$) and let $G$ be a group of orientation-preserving homeomorphisms of $M$ defined by some smoothness assumption which is stronger than $C^k$ but weaker than $C^{k+1}$.  Then, is it possible to topologize $G$ in such a way that it becomes a Polish group\footnote{That is, a group endowed with a separable completely metrizable topology under which the group multiplication and inversion maps are continuous.}?

More concretely, let $\mathcal{A}$ be an algebra of real-valued continuous functions defined on $M$ which contains all $C^1$ maps, and which satisfies the following closure properties: (1) whenever $f\in A$ and $f$ is everywhere positive then $\frac{1}{f}\in\mathcal{A}$, and (2) whenever $f\in \mathcal{A}$ and $g\in\Diff_+^1(M)$ then $f\circ g\in\mathcal{A}$.  Then for any integer $k\geq 1$, by iterating the chain rule, product rule, and inverse rule for derivatives, it is straightforward to verify that the collection\\

\begin{center} $G=\{f\in\Diff_+^k(M):f^{(k)}\in\mathcal{A}\}$
\end{center}
\vspace{.3cm}

\noindent comprises a subgroup of $\Diff_+^k(M)$.  We list some significant examples of such algebras $\mathcal{A}$ below:\\

\begin{itemize}
		\item $CBV(M)$, the space of all continuous functions of bounded total variation;
		\item $AC(M)$, the space of all absolutely continuous functions;
		\item $C^{0,\epsilon}(M)$, the space of all H\"older continuous functions with exponent $0<\epsilon<1$; and
		\item $\mbox{\rm Lip}(M)$, the space of all Lipschitz continuous functions.
\end{itemize}
\vspace{.3cm}

The groups $G$ corresponding to each space above are respectively denoted: $\Diff_+^{k+BV}(M)$; $\Diff_+^{k+AC}(M)$; $\Diff_+^{k+\epsilon}(M)$; and $\Diff_+^{k+Lip}(M)$.  In all four cases the group $G$ satisfies $\Diff_+^{k+1}(M)\leq G\leq \Diff_+^k(M)$, where both containments are strict, and as dense subsets.  So these groups $G$ are in some sense natural interpolations between the $C^k$ diffeomorphism groups, which are well-known to be Polish groups in their respective $C^k$ topologies.  In the cases of $AC(M)$ and $\mbox{\rm Lip}(M)$, one can also drop below $C^1$ regularity, and define the groups $\Homeo_+^{AC}(M)$ and $\Homeo_+^{Lip}(M)$, consisting of those maps which together with their inverse lie in $AC(M)$ or $\mbox{\rm Lip}(M)$ respectively.  In this way we obtain groups which interpolate between the Polish groups $\Homeo_+(M)$ and $\Diff_+^1(M)$, i.e. between $C^0$ and $C^1$.  Such ``intermediate smoothness'' conditions have risen to prominence in the study of one-dimensional dynamical systems, because such assumptions have sharp consequences on the dynamics of maps satisfying them.  This connection traces back at least to the classical Denjoy theorem, which asserts that $C^{1+BV}$ circle diffeomorphisms with irrational rotation number do not admit exceptional minimal invariant sets (see \cite{navas_2011a} for this and many other results of a similar flavor).

In \cite{cohen_kallman_2015a}, Kallman and the author investigated the Polishability problem for several of the groups mentioned above, and it turns out that $\Homeo_+^{Lip}(M)$, $\Diff_+^{1+\epsilon}(M)$, and $\Diff_+^{1+Lip}(M)$ are all examples of a peculiar phenomenon: they are continuum-cardinality groups whose algebraic structure precludes the existence of any Polish group topology whatsoever\footnote{The group $\Diff_+^{1+Lip}(M)$ is not mentioned explicitly in \cite{cohen_kallman_2015a}, but the arguments given for $\Diff_+^{1+\epsilon}(M)$ hold just as well for $\epsilon=1$, which implies the result described.}.  This puts them in somewhat curious company, alongside free groups and free abelian groups on continuum-many generators (\cite{dudley_1961a}); the automorphism group of the category algebra of $\mathbb{R}$ (\cite{hjorth_2000a}); the homeomorphism groups of the rationals and the irrationals (\cite{rosendal_2005a}); and very few other known examples (although this list seems to be growing recently, see \cite{ibarlucia_melleray_2016a}, \cite{mann_2017a}, \cite{paolini_shelah_2017a}).

Contrariwise, Solecki has proven in \cite{solecki_1999a} that the group $\Homeo_+^{AC}(I)$ of absolutely continuous interval homeomorphisms, which have absolutely continuous inverse, does indeed carry a Polish group topology, and moreover this topology is strictly finer than the natural Polish topology inherited from $\Homeo_+(I)$, and strictly coarser than the natural Polish topology on $\Diff_+^1(I)$.  We remark here that if a group of homeomorphisms of $M=I$ or $M=\mathbb{S}^1$ is sufficiently rich in the sense that it is \textit{locally moving}\footnote{A subgroup $G$ of $\Homeo_+(M)$ is called locally moving if whenever $U\subseteq M$ is open, there is a non-identity $f\in G$ whose support is contained in $U$.}, and if it carries a Polish group topology, then this topology is necessarily unique by results of Kallman (\cite{kallman_1986a}), and it generates the same Borel $\sigma$-algebra as given by the subspace topology inherited from $\Homeo_+(M)$ (see \cite{cohen_kallman_2015a} Lemma 1.5).

Our first main result may be viewed as an extension of Solecki's theorem for $\Homeo_+^{AC}(I)$ to higher degrees of smoothness:

\begin{thm}[see Theorem \ref{thm_main}] \label{thm_1}  Let $k\geq 1$ and $M=I$ or $M=\mathbb{S}^1$.  Then the group $\Diff_+^{k+AC}(M)$ admits a Polish group topology.  This topology is strictly finer than that inherited as a subspace of $\Diff_+^k(M)$ in its usual topology, and its restriction to $\Diff_+^{k+1}(M)$ is strictly coarser than the usual topology on $\Diff_+^{k+1}(M)$.
\end{thm}

Thus we furnish infinitely many new examples of Polish groups $\Diff_+^{k+AC}(M)$ (a sequence corresponding to $M=I$ and another to $M=\mathbb{S}^1$).  The results of Solecki, Kallman and the author, and those of the present work seem to suggest that absolute continuity is the ``right'' intermediate smoothness condition to impose as opposed to Lipschitz/H\"older conditions or bounded variation, at least from the viewpoint of topological group theory.  We once again emphasize this contrast with our secondary result, which builds upon the work of \cite{cohen_kallman_2015a}:

\begin{thm}[see Theorem \ref{thm_main_2}] \label{thm_2}  Let $M=I$ or $M=\mathbb{S}^1$.  Then there is no Polish group topology on $\Diff_+^{1+BV}(M)$.
\end{thm}

A summary of the Polishability results of \cite{solecki_1999a}, \cite{cohen_kallman_2015a}, and the present work is diagrammed in Figure 1.

\begin{figure}
\begin{tikzpicture}[scale=0.95,
		polish/.style={draw, thick, color=green, text=black, shape=rectangle, style=solid},
		npolish/.style={draw, thick, color=red, text=black, shape=rectangle, style=dotted}]

		\node[polish] (Homeo) at (0,0) {$\Homeo_+(M)$};
		\node[polish] (HomeoAC) at (0,-1) {$\Homeo_+^{AC}(M)$};
		\node[npolish] (HomeoLip) at (0,-2) {$\Homeo_+^{Lip}(M)$};
		\node[polish] (Diff1) at (0,-3) {$\Diff_+^1(M)$};
		\node[npolish] (Diff1ep) at (-3,-4.5) {\begin{tabular}{r}{\tiny $\epsilon\rightarrow0$} $\nearrow$ \\[.2cm] $\Diff_+^{1+\epsilon}(M)$\end{tabular}};
		\node[npolish] (Diff1BV) at (3,-4) {$\Diff_+^{1+BV}(M)$};
		\node[polish] (Diff1AC) at (3,-5) {$\Diff_+^{1+AC}(M)$};
		\node[npolish] (Diff1Lip) at (0,-6) {$\Diff_+^{1+Lip}(M)$};
		\node[polish] (Diff2) at (0,-7) {$\Diff_+^2(M)$};
		\node (Diff2ep) at (-3,-8.5) {\begin{tabular}{r}{\tiny $\epsilon\rightarrow0$} $\nearrow$ \\[.2cm] $\Diff_+^{2+\epsilon}(M)$\end{tabular}};
		\node (Diff2BV) at (3,-8) {$\Diff_+^{2+BV}(M)$};
		\node[polish] (Diff2AC) at (3,-9) {$\Diff_+^{2+AC}(M)$};
		\node (Diff2Lip) at (0,-10) {$\Diff_+^{2+Lip}(M)$};
		\node[polish] (Diff3) at (0,-11) {$\Diff_+^3(M)$};
		\node (Ellipsis1) at (0,-12) {$\vdots$};
		\node[polish] (Diffk) at (0,-13) {$\Diff_+^k(M)$};
		\node (Diffkep) at (-3,-14.5) {\begin{tabular}{r}{\tiny $\epsilon\rightarrow0$} $\nearrow$ \\[.2cm] $\Diff_+^{k+\epsilon}(M)$\end{tabular}};
		\node (DiffkBV) at (3,-14) {$\Diff_+^{k+BV}(M)$};
		\node[polish] (DiffkAC) at (3,-15) {$\Diff_+^{k+AC}(M)$};
		\node (DiffkLip) at (0,-16) {$\Diff_+^{k+Lip}(M)$};
		\node[polish] (Diffk+1) at (0,-17) {$\Diff_+^{k+1}(M)$};
		\node (Ellipsis2) at (0,-18) {$\vdots$};		
		\node [polish] (DiffInfty) at (0,-19) {$\Diff_+^\infty(M)$};
		
		
		\draw (Homeo) -- (HomeoAC) -- (HomeoLip) -- (Diff1) -- (Diff1BV) -- (Diff1AC) -- (Diff1Lip) -- (Diff2);
		\draw (Diff1) -- (Diff1ep) -- (Diff1Lip);
		\draw (Diff2) -- (Diff2BV) -- (Diff2AC) -- (Diff2Lip) -- (Diff3) -- (Ellipsis1) -- (Diffk);
		\draw (Diff2) -- (Diff2ep) -- (Diff2Lip);
		\draw (Diffk) -- (DiffkBV) -- (DiffkAC) -- (DiffkLip) -- (Diffk+1) -- (Ellipsis2) -- (DiffInfty);
		\draw (Diffk) -- (Diffkep) -- (DiffkLip);
\end{tikzpicture}
\caption{The descending chain of subgroups of $\Homeo_+(M^1)$ which satisfy ``intermediate'' smoothness conditions.  A solid green box enshrines a Polish group; while a dotted red box contains a group with no compatible Polish group topology.  As one proceeds down the chain of Polish groups, the associated topologies become strictly finer.}
\end{figure}
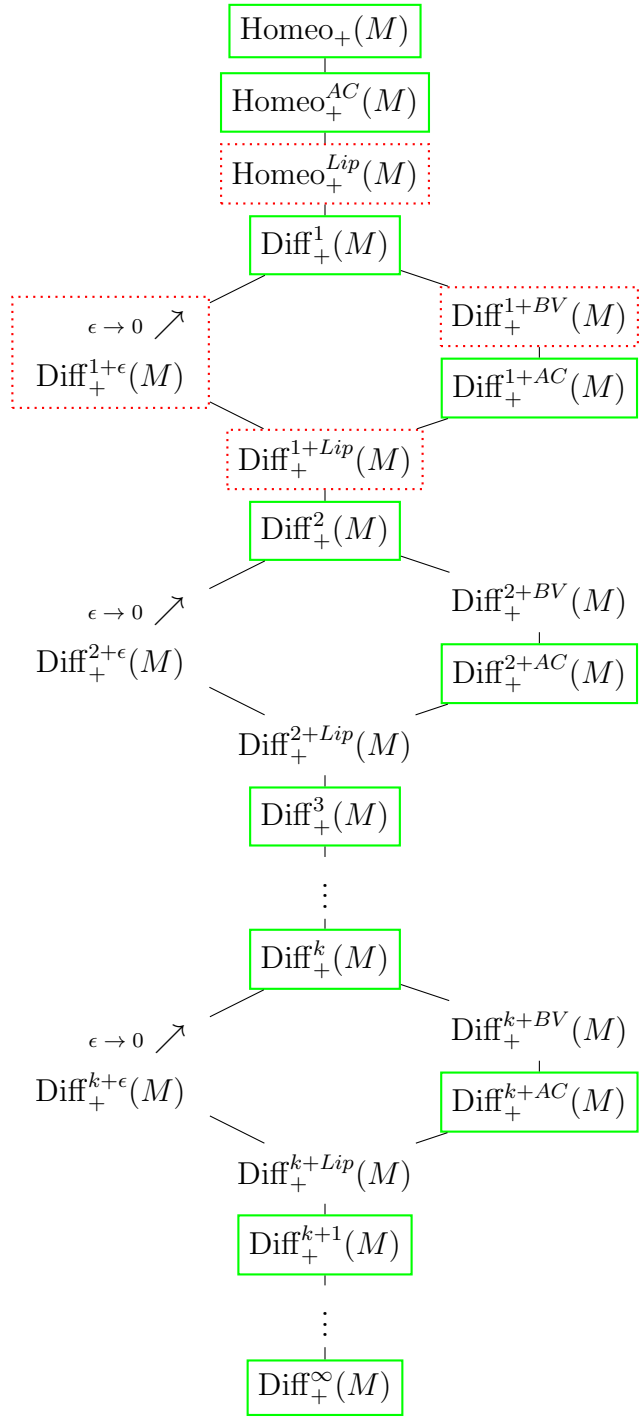

The fact that $\Diff_+^{1+BV}(M)$, $\Diff_+^{1+\epsilon}(M)$, and $\Diff_+^{1+Lip}(M)$ all fail to carry a Polish group topology may come as something of a surprise, since the smoothness classes used to define them ($CBV(M)$, $C^{0,\epsilon}(M)$, and $\mbox{\rm Lip}(M)$ respectively) are all Banach spaces, which suggests a natural assignment of a complete metric by mapping first derivatives into the Banach space where they reside, and taking the norm of the difference there.  (In fact this is essentially our strategy in the $\Diff_+^{k+AC}(M)$ case.)

In the cases of $\Diff_+^{1+\epsilon}(M)$ and $\Diff_+^{1+Lip}(M)$, we know in advance that such a strategy is doomed to fail because a Polish group topology on these groups does not exist.  On the other hand, $\Diff_+^{1+BV}(M)$ is a bit of an isolated case here: \textit{a priori}, our theorem above (non-existence of Polish group topology) does not necessarily imply that the strategy will fail for $\Diff_+^{1+BV}(M)$\textemdash since $CBV(M)$ fails to be separable (see Corollary \ref{cor_discrete_net}) the topology, were it group, would not be Polish anyway.  However, group multiplication does indeed fail to be continuous in this topology.  As an aside, we take the time to prove this (possibly folklore) fact:

\begin{prop}[see Proposition \ref{prop_discontinuity}]  Let $\Diff_+^{1+BV}(M)$ be topologized by the metric\\

\begin{center} $d(f,g)=\displaystyle\sup_{x\in M}d_M(f(x),g(x))+\|\log f'-\log g'\|_{BV}$
\end{center}
\vspace{.3cm}

\noindent where $d_M$ is the metric on $M$ and $\|\cdot\|_{BV}$ is the usual norm on the Banach space $CBV(M)$.  Then for each $g\in\Diff_+^{1+BV}(M)\backslash\Diff_+^{1+AC}(M)$, the left multiplication mapping $f\mapsto g\circ f$ is not continuous.
\end{prop}

In Section 3 we will recall the definition and properties of absolutely continuous maps and present our Polishability results for $\Diff_+^{k+AC}(M)$, and in Section 4 we will review maps of bounded variation and show non-Polishability of $\Diff_+^{1+BV}(M)$.  For background regarding absolute continuity and bounded variation, we found Leoni's book \cite{leoni_2009a} to be a wonderful resource, and we refer the reader there for additional reading.

We remark that historically, the notions of absolute continuity and bounded variation tend to arise in duality.  While functions of bounded variation satisfy certain very weak regularity properties, if one desires stronger conclusions, then the necessary strengthening of the bounded variation hypothesis often turns out to be exactly absolute continuity.  This trend is epitomized by the fact that functions of bounded variation are almost everywhere differentiable but need not satisfy the fundamental theorem of calculus; those functions which do so are precisely the absolutely continuous functions (see Section 3).  Similarly, for a bounded variation function $g$ and its translation $g_h$ by a real distance $h>0$, the total variation of the difference $g_h-g$ is bounded above by twice the variation of $g$, but need not tend to $0$ as $h\rightarrow0$; it turns out this variation tends to $0$ with $h$ if and only if $g$ is absolutely continuous (see \cite{wiener_young_1933a}).  Our Theorems \ref{thm_2} and \ref{thm_1} seem to echo this classical duality: the group $\Diff_+^{1+BV}(M)$ consists of maps which satisfy important and desirable dynamical properties, but the group cannot be given a nice natural topology; by refining to the subgroup $\Diff_+^{1+AC}(M)$ we achieve this goal.

\section{Preliminaries and notation}

For our purposes in this article, the difference between the interval and the circle is mostly inconsequential.  For this reason, we will choose some notational conventions that allow us to blur the distinction, and deal with both cases simultaneously.

We imagine $\mathbb{S}^1$ as the topological quotient space $I/E$, where $E$ is the equivalence relation on $I$ identifying $0$ and $1$.  We define the distance $d_{\mathbb{S}^1}$ in the circle via a natural parametrization, say, $d_{\mathbb{S}^1}([x],[y])=\|e^{2\pi ix}-e^{2\pi iy}\|$ for all (equivalence classes) $[x],[y]\in\mathbb{S}^1=I/E$, where the norm is taken in $\mathbb{C}$.  Whenever $x\in\mathbb{R}$ and $0\leq x\leq 1$, we adopt the habit of identifying $x$ with its equivalence class $[x]$, and understanding $x\in\mathbb{S}^1$ or $x\in I$ interchangeably.

To each homeomorphism $f:\mathbb{S}^1\rightarrow\mathbb{S}^1$ we associate its unique lift map $\tilde{f}:\mathbb{R}\rightarrow\mathbb{R}$, which satisfies: (1) $\tilde{f}(0)\in[0,1)$; (2) $\tilde{f}(x)+1=\tilde{f}(x+1)$ for all $x\in\mathbb{R}$; and (3) $f(x \mod 1)=\tilde{f}(x) \mod 1$ for all $x\in\mathbb{R}$.  We say such a homeomorphism $f$ is of class $C^k$ if and only if $\tilde{f}$ is of class $C^k$.  Inclusion of a circle homeomorphism $f$ in the class $C^k$ necessarily implies that $\tilde{f}^{(j)}(0)=\tilde{f}^{(j)}(1)$ for all $0\leq j\leq k$.  In general we define the derivatives $f^{(j)}(x)=\tilde{f}^{(j)}(x)$ for each $x\in\mathbb{S}^1$, each $f\in\Diff_+^k(\mathbb{S}^1)$, and each $1\leq j\leq k$.  Thus we imagine each $f^{(j)}$ as a member of $C(I)$ which takes the same value at $0$ and $1$.  It is straightforward to check that the usual chain rule, product rule, and inverse rule are still valid when we view the derivatives in this way.

Naturally, we consider a circle diffeomorphism $f$ to be of class $C^{k+AC}$ if $f^{(k)}$ is absolutely continuous, and of class $C^{k+BV}$ if $f^{(k)}$ is continuous and of bounded variation.

From now until the end of this article, unless we specifically indicate otherwise, we allow $M$ to denote either $I$ or $\mathbb{S}^1$.  It is well known that the diffeomorphism group $\Diff_+^k(M)$ is a Polish group when endowed with the metric\\

\begin{center} $d_{C^k}(f,g)=\displaystyle\sup_{x\in M}|d_M(f(x),g(x))|+\sum_{i=1}^k\|f^{(i)}(x)-g^{(i)}(x)\|$,
\end{center}
\vspace{.3cm}

\noindent and we make reference to this metric in several proofs.

\section{Polishability in the absolute continuity case}

Recall that a continuous function $f:I\rightarrow\mathbb{R}$ is said to be \textit{absolutely continuous} if for every $\epsilon>0$, there exists $\delta>0$ such that $\displaystyle\sum_{i=1}^m|f(b_i)-f(a_i)|<\epsilon$ whenever $0\leq a_1<b_1<...<a_n<b_n\leq 1$ satisfy $\displaystyle\sum_{i=1}^m(b_i-a_i)<\delta$.  Every $C^1$ or Lipschitz map is absolutely continuous.

The collection $AC(I)$ of all absolutely continuous real-valued functions on $I$ is closed under pointwise sums, differences, products, and reciprocals provided the function being reciprocated is everywhere positive.  If $u:M\rightarrow M$ is a Lipschitz (or $C^1$) homeomorphism, and $F\in AC(I)$, then $F\circ u\in AC(I)$.  As we mentioned in the introduction, $AC(I)$ consists of exactly those maps which obey the fundamental theorem of calculus; that is, if $F\in AC(I)$ then $F$ is almost everywhere differentiable and

\begin{center} $F(x)=F(0)+\int_0^x F'(x)$ for all $x\in I$.
\end{center}
\vspace{.3cm}

Conversely if a map satisfies the above relation then it is absolutely continuous.

By $\lambda$ we denote the Lebesgue measure on $\mathbb{R}$, or the normalized Lebesgue measure on $\mathbb{S}^1$ or on an interval $[a,b]\subseteq\mathbb{R}$, where the underlying space will be clear from the context.  We denote by $L^1([a,b])$ the space of (equivalence classes up to almost everywhere equality of) $\lambda$-integrable real-valued functions on $[a,b]$.  We denote the standard norm by $\|\cdot\|_{L^1}$, so\\

\begin{center} $\|F\|_{L^1}=\int_a^b|F|$
\end{center}
\vspace{.3cm}

\noindent whenever $F\in L^1([a,b])$.  Any time we employ a norm $\|\cdot\|$ without a subscript, it means the uniform norm applied to a continuous function $F\in C([a,b])$, i.e.\\

\begin{center} $\|F\|=\displaystyle\sup_{x\in[a,b]}|F(x)-x|$.
\end{center}
\vspace{.3cm}

We need the following lemma which can be found in much more general form as Theorem 3.54 and Corollary 3.57 in \cite{leoni_2009a}.

\begin{lemma}[Change of variables]  Let $g\in L^1([c,d])$ and let $u:[a,b]\rightarrow [c,d]$ be an increasing $C^1$ diffeomorphism.  Then $(g\circ u)u'\in L^1([a,b])$, and the change of variables formula holds:  $\int_c^d g=\int_a^b (g\circ u)u'$.
\end{lemma}

We must adapt the preceding lemma to respect our notational conventions for circle diffeomorphisms.

\begin{lemma}[Change of variables for circle maps] \label{lemma_substitution}  Let $g\in L^1(I)$ and let $u:\mathbb{S}^1\rightarrow\mathbb{S}^1$ be an orientation-preserving $C^1$ diffeomorphism.  Then $(g\circ u)u'\in L^1(I)$, and the change of variables formula holds:\\

\begin{center} $\int_I g=\int_I (g\circ u)u'$.
\end{center}
\end{lemma}

\begin{proof}  Let $\tilde{u}:\mathbb{R}\rightarrow\mathbb{R}$ denote the lift of $u$ with $u(0)\in[0,1)$ and set $p=\tilde{u}(0)$, $q=\tilde{u}^{-1}(1)$.  Then $\tilde{u}|_{[0,q]}:[0,q]\rightarrow[p,1]$ and $\tilde{u}|_{[q,1]}:[q,1]\rightarrow[1,p+1]$ are both increasing $C^1$ diffeomorphisms and thus by the previous lemma we have\\

\begin{align*}
\int_I g &= \int_0^p g+\int_p^1 g\\
&= \int_1^{p+1}g+\int_p^1 g\\
&= \int_q^1 (g\circ \tilde{u})u'+\int_0^q(g\circ\tilde{u})u'\\
&= \int_I (g\circ u)u'.
\end{align*}
\end{proof}

\begin{lemma} \label{lemma_sub_continuity}  Let $g\in L^1(I)$.  Then the mapping $u\mapsto (g\circ u)u'$, $\Diff_+^1(M)\rightarrow L^1(I)$ is continuous.
\end{lemma}

\begin{proof}  Fix $u\in\Diff_+^1(M)$, and let $\epsilon>0$.  Let $h$ be a continuous function so that $\|g-h\|_{L^1}<\epsilon/3$.  Using the continuity of $h$, find $\delta$ such that if $v\in\Diff_+^1(M)$ and $\|u'-v'\|<\delta$, then $\|(h\circ u)u'-(h\circ v)v'\|<\epsilon/3$.  For such $v$, we have $\|(h\circ u)u'-(h\circ v)v'\|_{L^1}<\epsilon/3$ as well, and thus applying Lemma \ref{lemma_substitution} we have\\

\begin{align*}
\|(g\circ u)u'-(g\circ v)v'\|_{L^1} &\leq \|(g\circ u)u'-(h\circ u)u'\|_{L^1}+\|(h\circ u)u'-(h\circ v)v'\|_{L^1}\\
& \indent   +\|(h\circ v)v'-(g\circ v)v'\|_{L^1}\\
&\leq \|g-h\|_{L^1}+\|(h\circ u)u'-(h\circ v)v'\|_{L^1}+\|h-g\|_{L^1}\\
&< \epsilon.
\end{align*}
\end{proof}

\begin{lemma} \label{mult_continuity}  The multiplication mapping $(f,a)\mapsto f\cdot a$, $L^1(I)\times C(I)\rightarrow L^1(I)$ is continuous.
\end{lemma}

\begin{proof}  Fix $f\in L^1(I)$ and $a\in C(I)$.  Let $\epsilon>0$ and choose $\delta>0$ small enough so that $(\|a\|+\delta)\delta+\delta\|f\|_{L^1}<\epsilon$.  Then if $(g,b)\in L^1(I)\times C(I)$ with $\|g-f\|_{L^1},\|b-a\|<\delta$, we have\\

\begin{align*}
\int_I|g\cdot b-f\cdot a| &\leq \int_I|g\cdot b-f\cdot b|+\int_I|f\cdot b-f\cdot a|\\
&= \int_I|g-f||b|+\int_I|f||b-a|\\
&\leq \|b\|\int_I|g-f|+\|b-a\|\int_I|f|\\
&\leq (\|a\|+\delta)\delta+\delta\|f\|_{L^1}<\epsilon.
\end{align*}
\end{proof}

Next we give a technical argument due to Solecki, which we have extracted from his proof of the Polishability of $\Homeo_+^{AC}(I)$ in \cite{solecki_1999a}.  We will apply the argument in a very similar fashion later, to show the convergence of a certain sequence in $L^1$.

\begin{lemma} \label{lemma_solecki}  Let $(f_n)$ be a sequence in $\Diff_+^1(M)$, and let $f_0\in\Diff_+^1(M)$ such that $f_n\rightarrow f_0$ and $f_n'\rightarrow f_0'$ uniformly.  Then for every $\epsilon>0$, there exists $\delta>0$ with the property that whenever $E\subseteq I$ is a Borel set with $\lambda(E)<\delta$ and $(f_{n_i})$ is a subsequence of $(f_n)$, then there is a further subsequence $(f_{n_{i_\ell}})$ with\\

\begin{center} $\lambda\left(f_0(E)\cup\bigcup_{\ell\geq 1}f_{n_{i_\ell}}(E)\right)<\epsilon$.
\end{center}
\end{lemma}

\begin{proof}  First check that since $f_n'\rightarrow f_0'$ uniformly, given any $\epsilon>0$, we can find $\delta>0$ with the property that whenever $E\subseteq I$ is a Borel set with $\lambda(E)<\delta$, we will have $\lambda(f_n(E))=\int_E f_n'<\epsilon$ for every $n\geq0$.

Fix $\epsilon>0$.  Choose $\delta$ so that $\lambda(f_0(E))<\epsilon/3$ whenever $E$ is a Borel set with $\lambda(E)<\delta$.  For each $\ell\geq 1$, find $\epsilon_\ell$ such that $\sum_\ell \epsilon_\ell<\epsilon/3$, and choose $\delta_\ell$ such that $\lambda(f_n(E))<\epsilon_\ell$ for each $n$, whenever $E$ is a Borel set with $\lambda(E)<\delta_\ell$.

Now let $E$ be a Borel set with $\lambda(E)<\delta$ and let $(f_{n_i})$ be any subsequence.  Let $(K_\ell)$ be an increasing sequence of compact subsets of $E$ such that $\lambda(E\backslash K_\ell)<\delta_\ell$ for each $\ell$.  Let $U\supseteq f_0(E)$ be an open set with $\lambda(U)<\epsilon/3$.  Since $f_n\rightarrow f_0$ uniformly, we have $f_{n_i}(K_\ell)\subseteq U$ for all sufficiently large $i$.  Passing now to a subsequence $(f_{n_{i_\ell}})$, we have that $f_{n_{i_\ell}}(K_\ell)\subseteq U$.  Thus we have\\

\begin{center}  $\lambda(f_{n_{i_\ell}}(E)\backslash U)\leq \lambda(f_{n_{i_\ell}}(E\backslash K_\ell))<\epsilon_\ell$
\end{center}
\vspace{.3cm}

\noindent and therefore\\

\begin{center}  $\lambda\left(f_0(E)\cup\bigcup_{\ell\geq 1}f_{n_{i_\ell}}(E)\right)\leq \lambda(f_0(E))+\lambda(U)+\sum_{n\geq 1}\lambda(f_{n_{i_\ell}}(E)\backslash U)<\epsilon$.
\end{center}
\end{proof}


We now define a new metric $d$ on $\Diff_+^{k+AC}(M)$ by\\

\begin{align*}
d(f,g) &= \displaystyle\sup_{x\in M}|d_M(f(x),g(x))|+\sum_{i=1}^k\||f^{(i)}(x)-g^{(i)}(x)\|+\int_I|f^{(k+1)}-g^{(k+1)}|\\
&= d_{C^k}(f,g)+\|f^{(k+1)}-g^{(k+1)}\|_{L^1}
\end{align*}
\vspace{.3cm}

\noindent for all $f,g\in\Diff_+^{k+AC}(M)$.  Plainly the topology on $\Diff_+^{k+AC}(M)$ generated by $d$ is strictly finer than that generated by $d_{C^k}$.

\begin{thm} \label{thm_main}  The group $\Diff_+^{k+AC}(M)$, equipped with the metric $d$, is a Polish group.
\end{thm}

\begin{proof}  First we will establish that $\Diff_+^{k+AC}(M)$ is topologically a Polish space, then we will verify that the group operations are continuous.\\

\noindent\textit{Separability.}  Clearly the metric $d$ is chosen so that the map $f\mapsto (f,f^{(k+1}))$, $\Diff_+^{k+AC}(M)\rightarrow \Diff_+^k(M)\times L^1(I)$ is an isometry.  Since the codomain is separable, $(\Diff_+^{k+AC}(I),d)$ is separable.\\

\noindent \textit{Completely metrizability.}  Since $\Diff_+^k(M)$ is Polish, there is a complete metric $\rho_{C^k}$ on $\Diff_+^k(M)$.  Define a metric $\rho$ on $\Diff_+^{k+AC}(M)$ by $\rho(f,g)=\rho_{C^k}(f,g)+\|f^{(k+1)}-g^{(k+1)}\|_{L^1}$.  It is clear that $\rho$ is compatible with $d$, and we claim $\rho$ is complete.  If $(f_n)$ is a $\rho$-Cauchy sequence in $\Diff_+^{k+AC}(M)$, then $(f_n)$ is a Cauchy sequence in $(\Diff_+^k(M),\rho_{C^k})$ and $(f_n^{(k+1)})$ is a Cauchy sequence in $L^1(I)$.  By the completeness of the former group and the latter Banach space, we can find $g\in\Diff_+^k(M)$ such that $f_n\rightarrow g$ in the $\rho_{C^k}$ metric, as well as $a^{k+1}\in L^1(I)$ such that $f_n^{(k+1)}\rightarrow a^{k+1}$ in $L^1$.  These convergences, together with the absolute continuity of $f_n$, imply that for every $x\in I$ we have\\

\begin{align*}
g^{(k)}(x)-g^{(k)}(0) &= \displaystyle\lim_{n\rightarrow\infty} f_n^{(k)}(x)-f_n^{(k)}(0)\\
&= \displaystyle\lim_{n\rightarrow\infty} \int_0^x f_n^{(k+1)}\\
&= \int_0^x a^{k+1}.
\end{align*}
\vspace{.3cm}

\noindent Therefore $(g^{(k)})'=a^{k+1}$ almost everywhere.  So $g^{(k)}$ is absolutely continuous, whence $f_n\rightarrow g\in\Diff_+^{k+AC}(M)$.\\

\noindent\textit{Continuity of inversion.}  By separability, it suffices to show that whenever $(f_n)$ is a sequence in $\Diff_+^{k+AC}(M)$ converging to $f_0\in\Diff_+^{k+AC}(M)$, then we have $f_n^{-1}\rightarrow f_0^{-1}$ (in the $d$ metric).  Since $d$ is stronger than $d_{C^k}$, and $\Diff_+^k(M)$ is already a topological group containing $\Diff_+^{k+AC}(M)$, we have that $d_{C^k}(f_n^{-1},f_0^{-1})\rightarrow0$ already.  So we need to show that $\|(f_n^{-1})^{(k+1)}-(f_0^{-1})^{(k+1)}\|_{L^1}\rightarrow 0$.  By considering the function $\id=f_n\circ f_n^{-1}$, and iterating the inverse rule, chain rule, and product rule (which work everywhere for the $C^1$ functions $f_n,f_n',...,f_n^{(k-1)}$ and almost everywhere for the absolutely continuous function $f_n^{(k)}$), for $n\geq0$ we compute that

\begin{center}  $(f_n^{-1})^{(k+1)}=[(f_n')^{-(2k+1)}\cdot P(f_n',f_n'',...,f_n^{(k)})-(f_n')^{-(k+2)}f_n^{(k+1)}]\circ f_n^{-1}$
\end{center}
\vspace{.3cm}

\noindent almost everywhere, where $P(z_1,...,z_k)$ is some polynomial in $k$ variables, independent of the choice of $f_n$.  Since $f_n^{(i)}\rightarrow f_0^{(i)}$ uniformly for each $1\leq i\leq k$, it is clear that the left-hand term $(f_n')^{-(2k+1)}\cdot P(f_n',f_n'',...,f_n^{(k)})\rightarrow (f_0')^{-(2k+1)}\cdot P(f_0',f_0'',...,f_0^{(k)})$ uniformly as well.  In addition, $f_n\rightarrow f_0$ uniformly implies that $f_n^{-1}\rightarrow f_0^{-1}$ uniformly.  Combining these facts, we see that\\

\begin{center} $[(f_n')^{-2k}\cdot P(f_n',f_n'',...,f_n^{(k)})]\circ f_n^{-1}\rightarrow [(f_0')^{-(2k+1)}\cdot P(f_0',f_0'',...,f_0^{(k)})]\circ f_0^{-1}$ uniformly.
\end{center}
\vspace{.3cm}

This uniform convergence implies convergence in $L^1$.  So to finish our argument, it remains only to see that $[(f_n')^{-(k+2)}f_n^{(k+1)}]\circ f_n^{-1}\rightarrow [(f_0')^{-(k+2)}f_0^{(k+1)}]\circ f_0^{-1}$ in $L^1$.  In turn, since $(f_n')^{-(k+1)}\rightarrow (f_0')^{-(k+1)}$ uniformly, by Lemma \ref{mult_continuity} it suffices to prove:

\begin{claim} $[(f_n')^{-1}f_n^{(k+1)}]\circ f_n^{-1}\rightarrow [(f_0')^{-1}f_0^{(k+1)}]\circ f_0^{-1}$ in $L^1$.
\end{claim}

\begin{proof}[Proof of claim]  We appeal to the following fact which can be found in \cite{rudin_1987a}, Exercise 17(b) with a hint: If $g_n\in L^1(I)$ ($n\geq 0$) satisfy \textit{(i)} $\|g_n\|_{L^1}\rightarrow \|g_0\|_{L^1}$ and \textit{(ii)} $g_n\rightarrow g_0$ pointwise almost everywhere, then $g_n\rightarrow g_0$ in $L^1$.

Taking $g_n=[(f_n')^{-1}f_n^{(k+1)}]\circ f_n^{-1}$, let us prove that conditions (i) and (ii) above hold for some subsequence of $(g_n)$.  For (i), observe that by Lemma \ref{lemma_substitution} (change of variables) we have\\

\begin{align*}
\|g_n\|_{L^1} &= \int_I (|f_n^{(k+1)}|\circ f_n^{-1})\cdot ((f_n')^{-1}\circ f_n^{-1})\\
&= \int_I |f_n^{(k+1)}|\\
&= \|f_n^{(k+1)}\|_{L^1}.
\end{align*}
\vspace{.3cm}

Since $\|f_n^{(k+1)}\|_{L^1}\rightarrow\|f_0^{(k+1)}\|_{L^1}$, so too does $\|g_n\|_{L^1}\rightarrow\|g_0\|_{L^1}$.

Next we show (ii) holds for a subsequence.  Let $\epsilon_1>0$ be arbitrary.  Since $(f_n^{-1})'\rightarrow (f_0^{-1})'$ uniformly, we may associate to $\epsilon_1$ a number $\delta>0$ with the property described in the statement of Lemma \ref{lemma_solecki}.

Since $f_n^{(k+1)}\rightarrow f_0^{(k+1)}$ in $L^1$, there is a subsequence of indices $(n_i)$ such that $(f_{n_i}^{(k+1)})$ converges pointwise almost everywhere to $f_0^{(k+1)}$ (\cite{rudin_1987a} Theorem 3.12).  Consequently $(f_{n_i}')^{-1}f_{n_i}^{(k+1)}\rightarrow(f_0')^{-1}f_0^{(k+1)}$ pointwise almost everywhere.  By Egorov's theorem (\cite{rudin_1987a} Chapter 3 Exercise 16 with a hint), we may find a Borel subset $E\subseteq I$ with $\lambda(E)<\delta$ and such that $(f_{n_i}')^{-1}(f_{n_i}^{(k+1)})$ converges uniformly to $(f_0')^{-1}f_0^{(k+1)}$ on $I\backslash E$.  By our choice of $\delta$ (using Lemma \ref{lemma_solecki}), we may pass to a further subsequence $(n_{i_\ell})$ so that if $A_1=f_0(E)\cup\bigcup_{\ell\geq 1}f_{n_{i_\ell}}(E)$, then $\lambda(A_1)<\epsilon_1$.

Now if $x\notin A_1$, then $f_{n_{i_\ell}}^{-1}(x)\notin E$ for any $\ell\geq0$, whence\\

\begin{center} $g_{n_{i_\ell}}(x)=[(f_{n_{i_\ell}}')^{-1}f_{n_{i_\ell}}^{(k+1)}]\circ f_{n_{i_\ell}}^{-1}(x)\rightarrow[(f_0')^{-1}f_0^{(k+1)}]\circ f_0^{-1}(x)=g_0(x)$.
\end{center}
\vspace{.3cm}

Thus $g_{n_{i_\ell}}$ converges to $g_0$ pointwise on $I\backslash A_1$, where $\lambda(A_1)<\epsilon_1$.

Let us now relabel $(g_{n_{i_\ell}})_{\ell\geq 1}$ as $(g_\ell^1)_{\ell\geq 1}$.  By repeating the argument above with $(g_\ell^1)$ in place of $(g_n)$, we can find a subsequence $(g_\ell^2)$ of $(g_\ell^1)$ which converges to $g_0$ pointwise on $I\backslash A_2$, where $\lambda(A_2)<\epsilon_2$, for arbitrary $\epsilon_2$.

Letting $\epsilon_p\rightarrow 0$ and iterating the argument, we may find sequences $(g_\ell^p)$ for $p\geq 1$, where $(g_\ell^{p+1})$ is a subsequence of $(g_\ell^p)$, and $g_\ell^p\rightarrow g_0$ pointwise on $I\backslash A_p$ with $\lambda(A_p)<\epsilon_p$.  Consider the diagonal sequence $(g_\ell^\ell)$.  If $x\notin\bigcap_p A_p$, then there is a fixed $p$ such that $x\notin A_p$, whence $g_\ell^\ell(x)\rightarrow g_0(x)$, because a tail of $(g_\ell^\ell)$ is a subsequence of $(g_\ell^p)$.  Since $\lambda\left(\bigcap_p A_p\right)=0$, we have that $g_\ell^\ell\rightarrow g_0$ almost everywhere.  This implies that $g^\ell_\ell\rightarrow g_0$ in $L^1$, by the fact mentioned above.

The argument we applied above to the sequence $(g_n)$ applies just as well to any of its subsequences; that is, every subsequence of $(g_n)$ admits a subsequence converging to $g_0$ in $L^1$.  Consequently $g_n\rightarrow g_0$ in $L^1$, completing the proof of the claim.
\end{proof}

Our arguments earlier, together with the claim, show that inversion is continuous with respect to the metric $d$.\\ 

\noindent \textit{Continuity of multiplication.}  Because we have already shown the continuity of inversion, if each right multiplication map $f\mapsto f\circ g$ ($g\in\Diff_+^{k+AC}(M)$) is continuous, then so is every left multiplication map $f\mapsto f^{-1}\mapsto f^{-1}\circ g^{-1}\mapsto (f^{-1}\circ g^{-1})^{-1}=g\circ f$ ($g\in\Diff_+^{k+AC}(M)$).  Applying a classical theorem of Montgomery (\cite{montgomery_1936a}), this suffices to conclude that the group multiplication is jointly continuous, since we have already shown that $\Diff_+^{k+AC}(M)$ is Polish as a topological space.  So we need only check that right multiplication by a fixed element $g$ is continuous.

We must show that if $f_n\rightarrow f_0$ in $\Diff_+^{k+AC}(M)$, then $d(f_n\circ g,f_0\circ g)=d_{C^k}(f_n\circ g,f_0\circ g)+\|(f_n\circ g)^{(k+1)}-(f_0\circ g)^{(k+1)}\|_{L^1}\rightarrow0$.  Since $d$ is stronger than $d_{C^k}$, and $\Diff_+^k(M)$ is already a topological group containing $\Diff_+^{k+AC}(M)$, we have that $d_{C^k}(f_n\circ g,f_0\circ g)\rightarrow0$ already.  So we need to show that $\|(f_n\circ g)^{(k+1)}-(f_0\circ g)^{(k+1)}\|_{L^1}\rightarrow 0$.  By iterating the chain and product rules we compute that $(f_n\circ g)^{(k+1)}$ is equal to

\begin{center} $(f_n^{(k+1)}\circ g)(g')^{k+1}+(f_n'\circ g)g^{(k+1)}+Q(f_n''\circ g, f_n^{(3)}\circ g,...,f_n^{(k)}\circ g,g',g'',...,g^{(k)})$
\end{center}
\vspace{.3cm}

\noindent where $Q(z_2,z_3,...,z_k,w_1,w_2,...,w_k)$ is some polynomial in $2k-1$ variables.  For short, let us refer to this function polynomial on the right-hand side of the expression above as $Q_n$.  Since $f_n^{(i)}\rightarrow f_0^{(i)}$ uniformly for $2\leq i\leq k$, we have $f_n^{(i)}\circ g\rightarrow f_0^{(i)}\circ g$ uniformly as well, and therefore $Q_n\rightarrow Q_0$ uniformly.  Thus $Q_n\rightarrow Q_0$ in $L^1$.  So to finish checking continuity of multiplication, we need to check that $(f_n^{(k+1)}\circ g)(g')^{k+1}\rightarrow (f_0^{(k+1)}\circ g)(g')^{k+1}$ and $(f_n'\circ g)g^{(k+1)}\rightarrow (f_0'\circ g)g^{(k+1)}$ in $L^1$.

By Lemma \ref{mult_continuity}, we have that $(f_n^{(k+1)}\circ g)(g')^{k+1}\rightarrow (f_0^{(k+1)}\circ g)(g')^{k+1}$ in $L^1$ if and only if $(f_n^{(k+1)}\circ g)\cdot g'\rightarrow (f_0^{(k+1)}\circ g)\cdot g'$ in $L^1$.  This convergence holds because by Lemma \ref{lemma_substitution},\\

\begin{align*}
\int_I|(f_n^{(k+1)}\circ g)\cdot g'- (f_0^{(k+1)}\circ g)\cdot g'| &= \int_I (|f_n^{(k+1)}- f_0^{(k+1)}|\circ g)\cdot g'\\
&= \int_I |f_n^{(k+1)}-f_0^{(k+1)}|\rightarrow 0.
\end{align*}
\vspace{.3cm}

For the second convergence, we have\\

\begin{align*}
\int_I|(f_n'\circ g)g^{(k+1)}-(f_0'\circ g)g^{(k+1)}| &= \int_I (|f_n'-f_0'|\circ g)\cdot |g^{(k+1)}|\\
&\leq \|f_n'-f_0'\|\int_I |g^{(k+1)}|\rightarrow 0.
\end{align*}

Thus multiplication on the right is continuous, and $\Diff_+^{k+AC}(M)$ is a Polish group.
\end{proof}


\section{Non-Polishability in the bounded variation case}

A map $F\in C([a,b])$ is said to be of \textit{bounded variation} if $\displaystyle\sup\displaystyle\sum_{i=1}^m|f(x_i)-f(x_{i-1})|$ is finite, where the supremum is taken over all finite partitions $a=x_0<x_1<...<x_m=b$.  We denote by $CBV([a,b])$ the subspace of $C([a,b])$ consisting of continuous maps of bounded variation.

$CBV([a,b])$ is closed under pointwise sums, differences, products, and reciprocals provided the function being reciprocated is everywhere positive.  If $u:[a,b]\rightarrow[a,b]$ is a homeomorphism, and $F\in CBV([a,b])$, then $F\circ u\in CBV([a,b])$.

If $F\in CBV([a,b])$ and $[c,d]\subseteq [a,b]$, we define the \textit{variation} of $F$ over $[c,d]$ to be\\

\begin{center} $V_c^d(F)=\displaystyle\sup\displaystyle\sum_{i=1}^m|F(x_i)-F(x_{i-1})|$
\end{center}
\vspace{.3cm}

\noindent where the supremum is taken over all finite partitions $c=x_0<x_1<...<x_n=d$.  $V_c^d$ returns finite values for every $F\in CBV([a,b])$, and satisfies the following properties:\\

\begin{itemize}
		\item if $e\in [c,d]$ then $V_c^d(F)=V_c^e(F)+V_e^d(F)$;
		\item if $K\in\mathbb{R}$ then $V_c^d(K\cdot F)=|K|\cdot V_c^d(F)$ and $V_c^d(F+K)=V_c^d(F)$; and
		\item if $u:[e,f]\rightarrow [c,d]$ is an increasing homeomorphism then $V_e^f(F\circ u)=V_c^d(F)$.
\end{itemize}
\vspace{.3cm}

If $c=a$ and $d=b$, then we drop subscript and superscript, and call $V(F)$ the \textit{total variation} of $F$.  $CBV([a,b])$ becomes a Banach space when equipped with the norm\\

\begin{center} $\|F\|_{BV}=|F(a)|+V(F)$.
\end{center}
\vspace{.3cm}

\begin{lemma}[Invariance of total variation under composition with circle maps] \label{lemma_var_invar}  Let $F\in CBV(I)$ and let $u:\mathbb{S}^1\rightarrow\mathbb{S}^1$ be an orientation-preserving homeomorphism.  Then $V(F)=V(F\circ u)$.
\end{lemma}

\begin{proof}  Let $\tilde{u}:\mathbb{R}\rightarrow\mathbb{R}$ denote the lift of $u$ with $u(0)\in[0,1)$ and set $p=\tilde{u}(0)$, $q=\tilde{u}^{-1}(1)$.  Then $\tilde{u}|_{[0,q]}:[0,q]\rightarrow[p,1]$ and $\tilde{u}|_{[q-1,0]}:[q-1,0]\rightarrow[0,p]$ are both increasing homeomorphisms.  Therefore\\

\begin{align*}
V(F) &= V_0^p(F)+V_p^1(F)\\
&= V_{q-1}^0(F\circ\tilde{u})+V_0^q(F\circ\tilde{u})\\
&= V_q^1(F\circ u)+V_0^q(F\circ u)\\
&= V(F\circ u).
\end{align*}
\end{proof}

A non-constant real-valued function $G$ is called \textit{singular} if $G'(x)=0$ almost everywhere.  The Lebesgue decomposition for functions of bounded variation states that every function $G\in CBV(I)$ decomposes uniquely into a sum $G=\alpha+\gamma$, where $\alpha$ is an absolutely continuous function and $\gamma$ is a continuous singular function.

We will require the following lovely theorem of Wiener and Young (\cite{wiener_young_1933a} Sections 2 and 3) for singular functions.

\begin{thm} \label{thm_wiener_young_1}  Let $g:\mathbb{R}\rightarrow \mathbb{R}$ be a singular function of bounded variation.  Let $g_h$ denote the function defined by $g_h(x)=g(x+h)$ ($x,h\in\mathbb{R}$).  Then for any $a<b$, for almost every $h$,\\

\begin{center} $V_a^b(g_h-g)=2V_a^b(g)$.
\end{center}
\end{thm}

The theorem above stands in stark contrast to the case of absolutely continuous functions\textemdash see Lemma \ref{lemma_sub_continuity}, bearing in mind that $V_a^b(f)=\int_a^b|f'|$ when $f$ is absolutely continous.  In fact, Wiener and Young were able to modify the construction of the classical Cantor devil's staircase function, to give examples of monotone functions where the relation in the previous theorem holds not just for {\it almost} every $h$, but for {\it every} $h$ (ibid., Section 6).

\begin{thm} \label{thm_wiener_young_2}  There exists a monotone function $g:\mathbb{R}\rightarrow\mathbb{R}$ with $g(x)=0$ for $x\leq 0$ and $g(x)=1$ for $x\geq 1$, such that if $g_h$ denotes the function defined by $g_h(x)=g(x+h)$ ($x\in\mathbb{R}$), then the following identity holds:\\

\begin{center} $V_a^b(g_h-g)=2$ for every $h\in\mathbb{R}$,
\end{center}
\vspace{.3cm}

\noindent whenever $[0,1]$ and $[h,1+h]$ are both contained in $[a,b]$.
\end{thm}

Following an observation of Adams (\cite{adams_1936a}), we can now deduce the following fact, which easily implies that $CBV(I)$ is a non-separable space in its given norm $\|\cdot\|_{BV}$.

\begin{cor} \label{cor_discrete_net}  Let $r>0$ be arbitrary, and let $B_r=\{F\in CBV(I):\|F\|_{BV}\leq r\}$ be the closed $r$-ball in $CBV(I)$.  Then there exists a subinterval $J\subseteq \mathbb{R}$, and an uncountable family $\{F_h\in CBV(I):h\in J\}$ such that\\

\begin{itemize}
		\item $F_h\in B_r$ for each $t\in U$;
		\item $F_h(0)=F_h(1)=0$ for each $t\in U$; and 
		\item $\|F_{h_1}-F_{h_2}\|_{BV}=2r$ for any distinct $h_1,h_2\in I$.
\end{itemize}
\end{cor}

\begin{proof}  Let $g:\mathbb{R}\rightarrow \mathbb{R}$ denote the monotone function described in Theorem \ref{thm_wiener_young_2}.  Define $F:I\rightarrow\mathbb{R}$ by:

\begin{equation*}
F(x)=\begin{cases}
\frac{1}{2}r\cdot g(6x-1), & \mbox{if $\frac{1}{6}\leq x\leq\frac{1}{3}$};\\
\frac{1}{2}r\cdot g(-6x+5), & \mbox{if $\frac{2}{3}\leq x\leq \frac{5}{6}$};\\
\frac{1}{2}r, & \mbox{if $\frac{1}{3}\leq x\leq\frac{2}{3}$};\\
0, & \mbox{if $x\leq\frac{1}{6}$ or $x\geq\frac{5}{6}$.}
\end{cases}
\end{equation*}
\vspace{.3cm}

Let $J=[-\frac{1}{6},\frac{1}{6}]$.  For each $h\in J$ we let $F_h:I\rightarrow\mathbb{R}$ be defined by $F_h(x)=F(x+h)$ for $x\in[\frac{1}{6}-h,\frac{5}{6}-h]$, and $F_h(x)=0$ otherwise.  Each $F_h$ is nondecreasing on $[0,\frac{1}{2}-h]$ and nonincreasing on $[\frac{1}{2}-h,0]$, and thus we have

\begin{center} $V(F_h)=V(F_0)=|F(\frac{1}{2})-F(0)|+|F(1)-F(\frac{1}{2})|=\frac{1}{2}r+\frac{1}{2}r=r$.
\end{center}
\vspace{.3cm}

Thus $F_h\in B_r$.  On the other hand, if $h_1,h_2\in J$ are distinct, then\\

\begin{align*}
V_0^{1/2}(F_{h_1}-F_{h_2}) &= V_0^{1/2}\left(\frac{1}{2}r\cdot g(6(x+h_1)-1)-g(6(x+h_2)-1)\right)\\
&= \frac{1}{2}r\cdot V_0^{1/2}(g(6(x+h_1)-1)-g(6(x+h_2)-1))\\
&= \frac{1}{2}r\cdot V_{-1}^2(g(x+h_1)-g(x+h_2))\\
&= \frac{1}{2}r\cdot V_{-1+h_2}^{1+h_2}(g_{h_1-h_2}-g)\\
&= \frac{1}{2}r\cdot 2=r.
\end{align*}
\vspace{.3cm}

Similarly we compute $V_{1/2}^2(F_{h_1}-F_{h_2})=r$, and thus $V(F_{h_1}-F_{h_2})=V_{0}^{1/2}(F_{h_1}-F_{h_2})+V_{1/2}^1(F_{h_1}-F_{h_2})=2r$.
\end{proof}

It is a folklore fact that $\Diff_+^1(I)$ is homeomorphic to the Banach space $\{F\in C(I):F(0)=0\}$ via the mapping $f\mapsto\log f'-\log f'(0)$ (we give a sketch of the argument in \cite{cohen_2017a}), while $\Diff_+^1(\mathbb{S}^1)$ surjects onto $\{F\in C(I):F(0)=F(1)=0\}$ via the same mapping, where the fibers mapping to each point are exactly the cosets of the group of rotations in $\Diff_+^1(\mathbb{S}^1)$.

Consider $f\in\Diff_+^1(M)$, and let $J=f'(I)\subseteq(0,\infty)$.  Since $\log$ is Lipschitz when restricted to $J$, it is easy to see from the definition of bounded variation that if $f'\in CBV(I)$ then so too is $\log f'\in CBV(I)$.  Conversely, since $\exp$ is Lipschitz when restricted to $\log f'(I)$, if $\log f'\in CBV(I)$ then $f'=\exp\log f'\in CBV(I)$.  Thus $f\in\Diff_+^{1+BV}(M)$ if and only if $\log f'-\log f'(0)\in CBV(I)$.

In particular, the mapping $f\mapsto \log f'-\log f'(0)$, viewed as a mapping from $\Diff_+^{1+BV}(M)$ to $CBV(I)$, is surjective, and bijective in case $M=I$.  This correspondence seems to suggest a natural choice\footnote{We have used $\|\log f'-\log g'\|_{BV}$ in our choice of metric, where we might alternatively have preferred $\|f'-g'\|_{BV}$ or $\|(\log f'-\log f'(0))-(\log g'-\log g'(0))\|_{BV}$.  These modifications lead to equivalent metrics.} of metric for the group $\Diff_+^{1+BV}(M)$:\\

\begin{center} $d(f,g)=\displaystyle\sup_{x\in M}d_M(f(x),g(x))+\|\log f'-\log g'\|_{BV}$
\end{center}
\vspace{.3cm}

\noindent for $f,g\in\Diff_+^{1+BV}(M)$.  Indeed, the restriction of this metric to $\Diff_+^{1+AC}(M)$ succeeds in inducing a Polish group topology on the latter group, since for absolutely continuous maps $V(f)=\int_I|f'|$ and one can check that this metric is equivalent to the one we studied earlier.  However, we will now observe that this metric does not define a group topology on $\Diff_+^{1+BV}(M)$, because group multiplication fails badly to be continuous when the factors are not of class $C^{1+AC}$.

\begin{prop} \label{prop_discontinuity}  For any $g\in\Diff_+^{1+BV}(M)\backslash\Diff_+^{1+AC}(M)$, the left multiplication map $f\mapsto g\circ f$, $\Diff_+^{1+BV}(M)\rightarrow\Diff_+^{1+BV}(M)$ is not continuous, where $\Diff_+^{1+BV}(M)$ is endowed with the metric $d(f,g)=\displaystyle\sup_{x\in M}d_M(X)+\|\log f'-\log g'\|_{BV}$.
\end{prop}

\begin{proof}  Let $G=\log g'\in CBV(I)$, and find its Lebesgue decomposition $G=\alpha+\gamma$ where $\alpha$ is absolutely continuous and $\gamma$ is continuous and singular.  Since $G$ is not absolutely continuous, $\gamma$ is not identically zero: let $(p,p+\delta)$ be a subinterval of $I$ where $\gamma|_{(p,p+\delta)}$ is non-constant and $0<\delta<p<p+\delta<1-\delta<1$.  Set $\epsilon=V_\delta^{1-\delta}(\gamma)>0$.

Applying Theorem \ref{thm_wiener_young_1}, let $h_n$ be a sequence of real numbers so that $|h_n|<\delta$, $h_n\rightarrow 0$, and $V_{\delta}^{1-\delta}(\gamma_{h_n}-\gamma)=2V_{\delta}^{1-\delta}(\gamma)$.  (Here, as before, $\gamma_{h_n}$ is defined by $\gamma_{h_n}(x)=\gamma(x+h_n)$ for $x\in\mathbb{R}$.)

For each integer $n\geq 1$ let $f_n$ be an element of $\Diff_+^{1+BV}(M)$ which maps $[p,p+\delta]$ linearly to $[p+h_n,p+\delta+h_n]$, and let us make our choices so that $\displaystyle\sup_{x\in M}d_M(f_n(x),x)\rightarrow 0$, and $\|\log f_n'\|_{BV}\rightarrow 0$.  For instance, if $M=\mathbb{S}^1$ we can take $f_n$ to be the rotation through an angle of $h_n\cdot 2\pi$; if $M=I$ it is easy to construct $f_n$ as a piecewise linear map.  Thus we have $f_n\rightarrow e$ in our chosen topology on $\Diff_+^{1+BV}(I)$.

On the other hand, by Theorem \ref{thm_wiener_young_1}, we have

\begin{align*}
\|\log (g\circ f_n)'-\log g'\|_{BV} &= \|G\circ f_n+\log f_n'-G\|_{BV}\\
&\geq \|\gamma\circ f_n-\gamma\|_{BV}-\|\alpha\circ f_n-\alpha\|_{BV}-\|\log f_n'\|_{BV}\\
&\geq V_{\delta}^{1-\delta}(\gamma\circ f_n-\gamma)-\left[|\alpha(f_n(0))-\alpha(0)|+V(\alpha\circ f_n-\alpha)+\|\log f_n'\|_{BV}\right]\\
&= V_{\delta}^{1-\delta}(\gamma_{h_n}-\gamma)-\left[|\alpha(f_n(0))-\alpha(0)|+\int_I|(\alpha'\circ f_n)\cdot f_n'-\alpha'|+\|\log f_n'\|_{BV}]\right]\\
&\geq 2\epsilon-\left[|\alpha(f_n(0))-\alpha(0)|+\int_I|(\alpha'\circ f_n)\cdot f_n'-\alpha'|+\|\log f_n'\|_{BV}\right].
\end{align*}
\vspace{.3cm}

Using Lemma \ref{lemma_sub_continuity}, and the continuity of $\alpha$, we can take the term being subtracted from $2\epsilon$ in the last line above to be less than $\epsilon$, for large enough $n$.  Therefore $\|g\circ f_n-g\|_{BV}\geq\epsilon$ for sufficiently large $n$, and $g\circ f_n\not\rightarrow g$.
\end{proof}

Lastly we will show that $\Diff_+^{1+BV}(M)$ has no compatible Polish group topology whatsoever.  The argument is based in part on the following key lemma which we present without proof here.  The lemma is stated explicitly and proved for the case $M=I$ as Lemma 2.3 in \cite{cohen_kallman_2015a}; for the case where $M=\mathbb{S}^1$, it may be easily extrapolated from the proof of Theorem 3.3.2 in \cite{hayes_1997a}.

\begin{lemma}  Let $M=I$ or $M=\mathbb{S}^1$.  Let $G$ be a locally moving subgroup of $\Homeo_+(M)$.  If $\tau$ is any Hausdorff group topology on $G$, then $\tau$ contains the compact-open topology.
\end{lemma}

\begin{thm} \label{thm_main_2}  There is no Polish group topology on $\Diff_+^{1+BV}(M)$.
\end{thm}

\begin{proof}  Consider the sets $A_n=\{f\in\Diff_+^{1+BV}(M): V(\log f')\leq n\}$, for $n\geq1$.  It is easy to see that $A_n=A_n^{-1}$ and $(A_n)^2\subseteq A_{2n}$, and we also have\\

\begin{center} $f\in A_n\leftrightarrow \forall p\geq 1~\forall q_0,q_1,...,q_p\in(\mathbb{Q}\cap I)~\sum_{i=1}^p|\log f'(q_i)-\log f'(q_{i-1})|\leq n$
\end{center}
\vspace{.3cm}

\noindent from which equivalence we can see that $A_n$ is a $G_\delta$ set in $\Diff_+^1(M)$.  This implies $A_n$ is Borel in $\Homeo_+(M)$ because the inclusion $\Diff_+^1(M)\rightarrow\Homeo_+(M)$ is continuous, which in turn implies that $\Diff_+^{1+BV}(M)$ is a Borel subgroup of $\Homeo_+(M)$ since $\Diff_+^{1+BV}(M)=\bigcup_n A_n$.

Assume for the sake of contradiction that $\Diff_+^{1+BV}(M)$ has some topology which makes it into a Polish group; by the previous lemma, the topology on $\Diff_+^{1+BV}$ is at least as fine as the usual subspace topology inherited from $\Homeo_+(M)$, and therefore the Borel sets in $\Homeo_+(M)$ are also Borel in $\Diff_+^{1+BV}(M)$.

In particular each $A_n$ is Borel in $\Diff_+^{1+BV}(M)$.  Since $\Diff_+^{1+BV}(M)=\bigcup_n A_n$, there is some integer $n_0\geq 1$ for which $A_{n_0}$ is non-meager in $\Diff_+^{1+BV}(M)$.  It follows from the Pettis theorem \cite{pettis_1950a} that $A_{2n_0}$ contains a neighborhood of identity in $\Diff_+^{1+BV}(M)$, since $A_{2n_0}\supseteq (A_{n_0})^2=A_{n_0}A_{n_0}^{-1}$.  Therefore there exists a sequence $(g_k)$ in $\Diff_+^{1+BV}(M)$ such that\\

\begin{center} $\Diff_+^{1+BV}(M)=\bigcup_k A_{2n_0}g_k$.
\end{center}
\vspace{.3cm}

By Corollary \ref{cor_discrete_net}, there exists an uncountable family $\{F_h\in CBV(I):h\in J\}$ such that $V(F_{h_1}-F_{h_2})=4n_0+1$ whenever $h_1,h_2\in J$ are distinct.  Let $\{f_h\in\Diff_+^{1+BV}(M):h\in J\}$ be such that $\log f_h'-\log f_h'(0)=F_h$.  Since $J$ is uncountable, there must be some integer $k_0$ and some pair of distinct indices $h_1,h_2\in J$ such that $f_{h_1}$ and $f_{h_2}$ both lie in $A_{2n_0}g_{k_0}$.  Therefore $f_{h_1}\circ f_{h_2}^{-1}= (f_{h_1}\circ g_{k_0}^{-1})\circ(f_{h_2}\circ g_{k_0}^{-1})^{-1} \in A_{2n_0}A_{2n_0}^{-1}\subseteq A_{4n_0}$, and we have\\

\begin{center} $V(\log (f_{h_1}\circ f_{h_2}^{-1})')\leq 4n_0$.
\end{center}
\vspace{.3cm}

On the other hand,

\begin{align*}
V(\log (f_{h_1}\circ f_{h_2}^{-1})') &= V(\log((f_{h_1}'\circ f_{h_2}^{-1})\cdot ((f_{h_2}')^{-1}\cdot f_{h_2}^{-1})))\\
&= V((\log f_{h_1}'-\log f_{h_2}')\circ f_{h_2}^{-1})\\
&= V((F_{h_1}-F_{h_2})\circ f_{h_2}^{-1}+f_{h_1}'(0)+f_{h_2}'(0))\\
&= V(F_{h_1}-F_{h_2})\\
&= 4n_0+1>4n_0,
\end{align*}
\vspace{.3cm}

\noindent a contradiction.
\end{proof}

\end{document}